\documentclass[12pt, reqno]{amsart}
\usepackage{amsmath, amsthm, amscd, amsfonts, amssymb, graphicx, color}

\newtheorem{theorem}{Theorem}[section]
\newtheorem{definition}{Definition}[section]
\newtheorem{lemma}{Lemma}[section]

\newtheorem{corollary}{Corollary}[section]
\newtheorem{remark}{Remark}[section]
\newtheorem{example}{Example}[section]
\numberwithin{equation}{section}

\begin{document}
\setcounter{page}{1}

\title[Fixed points of multivalued contractions]
{Fixed points of multivalued contractions via generalized class of simulation functions}

\author[Deepesh Kumar Patel]
{Deepesh Kumar Patel}

\address{Deepesh Kumar Patel \newline
Department of Mathematics, Visvesvaraya National Institute of Technology, Nagpur-440010, India}
\email{deepesh456@gmail.com}



\begin{abstract}
In this paper, considering a wider class of simulation functions some fixed point results for multivalued mappings in $\alpha$-complete metric spaces have been presented. Results obtained in this paper extend and generalize some well-known fixed point results of the literature. Some examples and consequence are given to  illustrate the usability of the theory.
\end{abstract} \maketitle

\section{Introduction}
\label{Introduction}
It is well-known that the theory of multivalued mappings has applications in control theory, convex optimization, differential equations, and economics.  Nadler \cite{Nadler1969} generalized the wellknown Banach contraction principle to multivalued mappings which became a great source of inspiration for researchers working in metric fixed point theory. There have been many attempts to generalize this result in metric  and other spaces by many authors. Some notable generalizations have been obtained in \cite{AlikhaniGopalMiandaraghRezapourShahzad2013, AydiAbbasVetro2012, BerindeBerinde2007, Ciric2009, CovitzNadler1970, KlimWardowski2007,  Markin1968, MizoguchiTakahashi1989, ShuklaGopalMoreno2017, Suzuki2008JMAA}.  In a recent work, Khojasteh \textit{et al.} \cite{KhojastehShuklaRadenovic} introduced the notion of $\mathcal{Z}$-contraction using a class of control functions called simulation functions and unified several results of the literature on single-valued mappings. Olgun \textit{et al.} \cite{OlgunBicer2016} obtained some fixed point results for generalized $\mathcal{Z}$-contractions. Further, de-Hierro \textit{et al.} \cite{HierroKarapinar} enlarged the class of simulation functions for a pair of mappings and obtained some coincidence point theorems.

In \cite{SametVetroVetro2012}, Samet \textit{et al.} introduced the concept of $\alpha$-admissibility which is interesting since it does not require the contraction conditions or the contraction type conditions to hold for every pair of points in the domain unlike the BCP. It also includes the case of discontinuous mappings.

There is now massive growth in the literature dealing with fixed point problems via $\alpha$-admissible mappings (cf. \cite{Karapinar2012, KarapinarKumamSalimi2013}). In \cite{Karapinar2016}, Karapinar introduced the notion of $\alpha$-admissible $\mathcal{Z}$-contraction and generalized the results of Samet \textit{et al.} \cite{SametVetroVetro2012} and Khojasteh \textit{et al.} \cite{KhojastehShuklaRadenovic}. Recently, Radenovic and Chandok \cite{RadenovicChandok} (see also Liu \textit{et al.} \cite{LiuAnsariChandokRadenovic2018}) enlarged the class of simulation functions and generalized the results obtained in \cite{HierroKarapinar, OlgunBicer2016}.

Motivated by the results of \cite{Karapinar2016} and \cite{RadenovicChandok}, in this article we widen the class of $\alpha$-admissible mapping for multivalued mappings and define $\alpha$-admissible $\mathcal{Z}$-contractive multivalued mappings and $\alpha$-admissible generalized $\mathcal{Z}$-contractive multivalued mappings. Subsequently, we obtain some fixed point results for these mappings. Some useful examples and consequence are also presented to illustrate the usability of the obtained results.

\section{Preliminaries}
%
The aim of this section is to present some notions and results used in the paper.
For a nonempty set $X$, let $\mathcal{P}(X)$ denotes the power set of $X$. If $(X,d)$ is a metric space, then let
\begin{enumerate}
\item[] $\mathcal{N}(X) = \mathcal{P}(X)-\{\emptyset\}$,
 \item[] $\mathcal{CB}(X)=\{A\in N(X)$ : $A$ is closed and bounded$\}$,
\item[] $d(A,B)= \inf\{d(a,b): a\in A$ and $b\in B\}$,
\item[] $d(a,B)=\inf_{b\in B}d(a,b)$, $a\in X$,
\item[] $\mathcal{H}(A,B)= \max\left\lbrace\sup\limits_{a\in A}d(a,B), \sup\limits_{b\in B}d(b,A)\right\rbrace$.
\end{enumerate}

The notion of $\alpha$-admissible and triangular $\alpha$-admissible mappings were introduced by Samet \textit{et al.} \cite{SametVetroVetro2012} and Karapinar \textit{et al.} \cite{KarapinarKumamSalimi2013}, respectively as follows.
\begin{definition}\rm 
Let $\alpha : X \times X \to  [0,\infty)$. A self-mapping $T:X\to X$ is called $\alpha$-admissible if the following condition holds:
\begin{equation*}
x,y \in X,~~\alpha(x,y) \geq 1 \Longrightarrow \alpha(Tx,Ty) \geq 1.
\end{equation*}
Moreover, a self-mapping $T$ is called triangular $\alpha$-admissible if $T$ is $\alpha$-admissible and
\begin{equation*}
x,y,z \in X,~~\alpha(x,z) \geq 1~~\text{and}~~\alpha(z,y)\geq 1 \Longrightarrow \alpha(x,y) \geq 1.
\end{equation*}
\end{definition}

Further, Asl \textit{et al.} \cite{AslRezapourShahzad2012} introduced the concept of an $\alpha_*$-admissible mapping which is a multivalued version of the $\alpha$-admissible mapping provided in \cite{SametVetroVetro2012}.

\begin{definition} \cite{AslRezapourShahzad2012}  \rm
Let $X$ be a nonempty set, $T:X\to \mathcal{N}(X)$ and $\alpha : X \times X \to  [0,\infty)$ be
two mappings. We say that $T$ is $\alpha_*$-admissible if the following condition holds:
\begin{equation*}
x,y \in X,~~\alpha(x,y) \geq 1 \Longrightarrow \alpha_*(Tx,Ty) \geq 1,
\end{equation*}
where $\alpha_*(Tx,Ty):=\inf\{\alpha(a,b) ~\vert ~a\in Tx, b\in Ty\}$.
\end{definition}

On the other hand, Mohammadi \textit{et al.} \cite{MohammadiRezapourNaseer2013} extended the concept of an $\alpha_*$-admissible mapping to $\alpha$-admissible as follows.

\begin{definition} \cite{MohammadiRezapourNaseer2013} \rm
Let $X$ be a nonempty set, $T:X\to \mathcal{N}(X)$ and $\alpha : X \times X \to  [0,\infty)$ be two given mapping. Then $T$ is said to be an $\alpha$-admissible if, whenever for each $x\in X$ and $y\in Tx$ 
\begin{equation*}
\alpha(x, y) \geq 1 \Longrightarrow \alpha(y, z) \geq 1,~~\text{for all}~z\in Ty.
\end{equation*}
\end{definition}

\begin{remark} \rm
It is clear that $\alpha_*$-admissible mapping is also $\alpha$-admissible, but the converse may not be true (see Example 15 of \cite{MinakAcarAltun2013}).
\end{remark}

\begin{definition}  \cite{HussainKutbiSalimi2014} \rm
Let $(X,d)$ be a metric space and $\alpha : X \times X \to  [0,\infty)$. The metric space $(X,d)$ is said to be $\alpha$-complete if and only if every Cauchy sequence $\{x_n\}$ with $\alpha(x_n, x_{n+1})\geq 1$ for all $n \in \mathbb{N}$ converges in $X$. 
\end{definition}

\begin{remark} \cite{HussainKutbiSalimi2014} \rm
If $X$ is complete metric space, then $X$ is also $\alpha$-complete metric space. But
the converse is not true.
\end{remark}

\begin{definition} \cite{KutbiSintunavarat2015} \rm 
Let $(X,d)$ be a metric space, $\alpha : X \times X \to  [0,\infty)$ and $T: X \to \mathcal{CB}(X)$ be two given mappings. Then $T$ is said to be an $\alpha$-continuous multivalued mapping on $(\mathcal{CB}(X), \mathcal{H})$ if, for all sequences $\{x_n\}$ with $x_n \xrightarrow[]{d} x\in X$ as $n\to \infty$, and $\alpha(x_n,x_{n+1})\geq 1$ for all $n \in \mathbb{N}$, we have $Tx_n \xrightarrow[]{\mathcal{H}} Tx$ as $n \to \infty$, that is,
\begin{equation*}
\lim\limits_{n\to\infty}d(x_n,x)=0~\text{and}~\alpha(x_n,x_{n+1})\geq 1~\text{for all}~n\in \mathbb{N} ~ \Longrightarrow ~\lim\limits_{n\to\infty}\mathcal{H}(Tx_n,Tx)=0.
\end{equation*}
\end{definition}

\begin{remark}  \rm
The continuity of $T$ implies the $\alpha$-continuity of $T$, for all mappings $\alpha$. In general, the converse is not true (see in Example 2.2 \cite{KutbiSintunavarat2015}).
\end{remark}

In \cite{KhojastehShuklaRadenovic}, Khojasteh \textit{et al.} defined a new class of  contraction mapping using the following class of simulation functions.

\begin{definition}\cite{KhojastehShuklaRadenovic} \rm
A simulation function is a mapping $\zeta : [0,\infty)^2 \to \mathbb{R}$ satisfying the following conditions:
\begin{enumerate}
\item[($\zeta 1$)] $\zeta(0,0)=0$,
\item[($\zeta 2$)] $\zeta (t, s) < s-t$ for all $t, s > 0$,
\item[($\zeta 3$)] if $\{t_n\}$ and $\{s_n\}$ are sequences in $(0,\infty)$ such that $\lim\limits_{n\to\infty} t_n = \lim\limits_{n\to\infty} s_n =l \in (0,\infty) $ then
$\lim\limits_{n\to\infty}\sup\zeta (t_n, s_n) < 0.$
\end{enumerate}
\end{definition}
Argoubi \textit{et al.}\cite{ArgoubiSametVetro2015} slightly modified the definition  of simulation function  by withdrawing the condition ($\zeta 1$).
\begin{definition}\rm \label{CGSimulationfunction}
A simulation function is a mapping $\zeta : [0,\infty)^2 \to \mathbb{R}$ satisfying the following:
\begin{enumerate}
\item[(i)] $\zeta (t, s) < s-t$ for all $t, s > 0$,
\item[(ii)] if $\{t_n\}$ and $\{s_n\}$ are sequences in $(0,\infty)$ such that $\lim\limits_{n\to\infty} t_n = \lim\limits_{n\to\infty} s_n > 0$ and $t_n < s_n$, then $\lim\limits_{n\to\infty}\sup \zeta(t_n, s_n) < 0.$
\end{enumerate}
\end{definition}
Let $\mathcal{Z}$ denotes the family of all simulation functions.
For examples and related results on simulation functions, one may refer to \cite{ AbbasLatif2016, ChenTang2016, HierroKarapinar,  HieroSamet, Karapinar2016, KarapinarKhojasteh, KhojastehShuklaRadenovic, KomalKumamGopal2016, KumamGopalBudhia2016, MongkolkehaChoKumam2017, NastasiVetro2015, Samet2015, TchierVetro2016}.
\vspace{.1cm}

\begin{definition}  \cite{Karapinar2016} \label{KarapniarDef} \rm
Let $T$ be a self-mapping on a metric space $X$ endowed with metric $d$. Let $\alpha : X \times X \to  [0,\infty)$ be such that
\begin{equation}\label{karapinarcont}
\zeta(\alpha(x, y) d (Tx, Ty), d (x, y))\ge 0 , ~~\text{for all}~~ x,y\in X.
\end{equation}
Then $T$ is called an $\alpha$-admissible $\mathcal{Z}$-contraction with respect to $\zeta$, where $\zeta \in \mathcal{Z}$.
\end{definition}

Recently, in \cite{RadenovicChandok}, Radenovic and Chandok and in \cite{LiuAnsariChandokRadenovic2018}, Liu \textit{et al.} enlarged the class of simulation functions and obtained some coincidence and common fixed point results.

\begin{definition} \cite{Ansari2014} \label{Cclassfunction} \rm
A mapping $G : [0,\infty)^2 \to \mathbb{R}$ is called $C$-class function if it is
continuous and satisfies the following conditions:
\begin{enumerate}
\item[(i)]  $G(s, t) \leq s$,
\item[(ii)] $G(s, t) = s$ implies that either $s = 0$ or $t = 0$, for all $s, t \in [0,\infty)$.
\end{enumerate}
\end{definition}

\begin{definition} \cite{LiuAnsariChandokRadenovic2018, RadenovicChandok} \label{CGSimulationfunction} \rm
A $C_G$-simulation function is a mapping $\zeta : [0,\infty)^2 \to \mathbb{R}$ satisfying the
following:
\begin{enumerate}
\item[(a)] $\zeta (t, s) < G(s, t)$ for all $t, s > 0$, where $G : [0,\infty)^2 \to \mathbb{R}$  is a $C$-class function,
\item[(b)] if $\{t_n\}$ and $\{s_n\}$ are sequences in $(0,\infty)$ such that $\lim\limits_{n\to\infty} t_n = \lim\limits_{n\to\infty} s_n > 0$ and
$t_n < s_n$, then $\lim\limits_{n\to\infty}\sup \zeta(t_n, s_n) < C_G.$
\end{enumerate}
\end{definition}
\begin{definition}  \cite{LiuAnsariChandokRadenovic2018, RadenovicChandok} \label{CGproperty} \rm
A mapping $G : [0,\infty)^2 \to \mathbb{R}$ has a property $C_G$, if there exists a $C_G \geq 0$ such that
\begin{enumerate}
\item[(i)]  $G(s, t) > C_G$ implies $s > t$,
\item[(ii)] $G(t, t) \le C_G$ for all $t \in [0,\infty)$.
\end{enumerate}
\end{definition}

Let $\mathcal{Z}_G$ denotes the family of all $C_G$-simulation functions $\zeta : [0,\infty)^2 \to \mathbb{R}.$
We state the following definitions by taking $g = I$ (identity mapping) in Definitions 2.1 and 2.2 in \cite{RadenovicChandok}.
\begin{definition} \cite{RadenovicChandok}\label{ZG-contraction} \rm
Let $(X, d)$ be a metric space and $T : X \to X$ be self-mappings.
The mapping $T$ is called a $\mathcal{Z}_G$-contraction if there exists $\zeta\in \mathcal{Z}_G$ such that
\begin{equation} \label{RadnovicContrA}
\zeta(d(Tx,Ty), d(x,y))\geq C_G
\end{equation}
for all $x, y \in X$ with $x \ne y.$
\end{definition}
If  $C_G = 0$, then we get $\mathcal{Z}$-contraction defined in \cite{KhojastehShuklaRadenovic}.

\begin{definition} \cite{RadenovicChandok}\label{GeneralizedZG-contraction} \rm
Let $(X, d)$ be a metric space and $T : X \to X$ be self-mappings.
The mapping $T$ is called a generalized $\mathcal{Z}_G$-contraction if there exists $\zeta\in \mathcal{Z}_G$ such that
\begin{equation} \label{GeneralizedZG-contraction}
\zeta\left(d(Tx,Ty), \max\left\lbrace d(x,y), d(x,Tx), d(y,Ty), \frac{d(x,Ty)+d(y,Tx)}{2} \right\rbrace\right)\geq C_G
\end{equation}
for all $x, y \in X$ with $x \ne y.$
\end{definition}
If $C_G = 0$, then we get $\mathcal{Z}$-contraction defined in \cite{OlgunBicer2016}.

\begin{lemma}  \cite{RadenovicChandok} \label{LemmaA} \rm
Let $(X, d)$ be a metric space and $\{x_n\}$ be a sequence in $X$
 such that $\lim\limits_{n\to \infty}d(x_n,x_{n+1})=0.$ If $\{x_n\}$ is not Cauchy then there exists $\varepsilon>0$ and two subsequences $\{x_{m(k)}\}$ and $\{x_{n(k)}\}$ of $\{x_n\}$ where $n(k)>m(k)>k$ such that
 \begin{eqnarray*}
 &&\lim\limits_{k\to\infty} d(x_{m(k)},x_{n(k)})=\lim\limits_{k\to\infty} d(x_{m(k)},x_{n(k)+1})=\lim\limits_{k\to\infty} d(x_{m(k)-1},x_{n(k)})=\varepsilon\\
&& ~\text{and}~\lim\limits_{k\to\infty} d(x_{m(k)-1},x_{n(k)+1})=\lim\limits_{k\to\infty} d(x_{m(k)+1},x_{n(k)+1})=\varepsilon.
 \end{eqnarray*}
\end{lemma}

\section{Main results}
\label{Main results}
\begin{definition}  \rm
Let $X$ be a nonempty set, $T:X\to \mathcal{N}(X)$ and $\alpha : X \times X \to  [0,\infty)$ be two mappings. Then $T$ is said to be triangular $\alpha_*$-admissible if $T$ is $\alpha_*$-admissible and
\begin{equation*}
\alpha(x,y) \geq 1 ~~\text{and}~~ \alpha_*(Tx, Ty) \geq 1 \Longrightarrow \alpha(x,z)\geq 1,~~~\forall~ z\in Ty.
\end{equation*}
\end{definition}

\begin{definition}  \rm
Let $X$ be a nonempty set, $T:X\to \mathcal{N}(X)$ and $\alpha : X \times X \to  [0,\infty)$ be two mappings. Then $T$ is said to be triangular $\alpha$-admissible if $T$ is $\alpha$-admissible and
\begin{equation*}
\alpha(x,y) \geq 1 ~~\text{and}~~ \alpha(y, z) \geq 1 \Longrightarrow \alpha(x,z)\geq 1,~~~\forall~ z\in Ty.
\end{equation*}
\end{definition}
A triangular $\alpha_*$-admissible mapping is also triangular $\alpha$-admissible, but the converse may not be true.

\begin{lemma}  \label{Lemma3a}
Let $T:X\to \mathcal{N}(X)$ be a triangular $\alpha$-admissible mapping. Assume that there exist $x_0 \in X$ and $x_1\in Tx_0$ such that $\alpha(x_0,x_1)\geq 1$. Then for a sequence $\{x_n\}$ such that $x_{n+1}\in Tx_n$, we have $\alpha(x_n,x_m)\geq 1$ for all $m,n \in \mathbb{N}$ with $n<m$.
\end{lemma}

\begin{proof}
Since there exist $x_0 \in X$ and $x_1 \in Tx_0$ such that $\alpha(x_0,x_1)\geq 1$, then by the $\alpha$-admissibility of $T$, we have $\alpha(x_1,x_2)\geq 1$. By continuing this process, we get $\alpha(x_n,x_{n+1})\geq 1$ for all $n\in \mathbb{N}\cup \{0\}$. Suppose that $n<m$. Since $\alpha(x_n,x_{n+1})\geq 1$ and $\alpha(x_{n+1},x_{n+2})\geq 1$, then using the triangular $\alpha$-admissibility  of $T$ we have $\alpha(x_n, x_{n+2})\geq 1$. Again, since $\alpha(x_n, x_{n+2})\geq 1$ and $\alpha(x_{n+2}, x_{n+3})\geq 1$, then we deduce $\alpha(x_n, x_{n+3})\geq 1$. By continuing this process, we get $\alpha(x_n, x_m)\geq 1$.
\end{proof}

\begin{definition}\rm
Let $(X, d)$ be a metric space and $T : X \to \mathcal{CB}(X)$. The mapping $T$ is said to be a  $\alpha$-admissible $\mathcal{Z}_G$-contractive multivalued mapping if there exists $\zeta\in \mathcal{Z}_G$ and $\alpha : X \times X \to  [0,\infty)$ such that
\begin{equation}\label{alphaAdmZGContraction}
 \zeta\big(\alpha(x, y) \mathcal{H}(Tx, Ty), d(x, y)\big)\geq C_G
\end{equation}
for all $x, y \in X$ with $x \neq y.$
\end{definition}

\begin{definition}\rm \label{AlphaAdmGeneralizedZG-contraction}
Let $(X, d)$ be a metric space and $T : X \to \mathcal{CB}(X)$. We say $T$ is $\alpha$-admissible generalized $\mathcal{Z}_G$-contractive multivalued mapping if there exists $\zeta\in \mathcal{Z}_G$ and $\alpha : X \times X \to  [0,\infty)$ such that
\begin{equation}\label{AlphaAdmGeneralizedZG-contraction}
 \zeta(\alpha(x, y) \mathcal{H}(Tx, Ty), M(x, y))\ge C_G
\end{equation}
for all $x, y \in X$ with $x \neq y$, where
\begin{equation*}
M(x,y)=\max\left\lbrace d(x,y), d(x,Tx), d(y,Ty), \frac{d(x,Ty)+d(y,Tx)}{2} \right\rbrace.
\end{equation*}
\end{definition}
The following is the first main result in this paper.
\begin{theorem} \label{TheoremAlpha1}	
Let $(X, d)$ be a metric space, $T : X \to \mathcal{CB}(X)$ be a  $\alpha$-admissible generalized $\mathcal{Z}_G$-contractive multivalued mapping. Suppose the following conditions hold:
\begin{enumerate}
\item[(i)] $(X, d)$ is an $\alpha$-complete metric space,
\item[(ii)] there exist $x_0 \in X$ and $x_1\in Tx_0$ such that $\alpha(x_0,x_1) \geq 1$,
\item[(iii)] $T$ is a triangular $\alpha$-admissible,
\item[(iv)] $T$ is an $\alpha$-continuous multivalued mapping.
\end{enumerate}
Then $T$ has a fixed point.
\end{theorem}
\begin{proof}
From condition (ii), we have $x_0 \in X$ and $x_1\in Tx_0$ such that $\alpha(x_0,x_1) \geq 1$. If $x_0 = x_1$ or $x_1 \in Tx_1$, then $x_1$ is a fixed point of $T$ and we are done. Assume that $x_1 \not\in Tx_1$. Now, since $T$ is a mapping from $X$ to $\mathcal{CB}(X)$,  so we can choose a $x_2 \in Tx_1$ such that
$$d(x_1,x_2)\leq \mathcal{H}(Tx_0,Tx_1).$$
Again we can choose a point $x_3 \in Tx_2$ such that
$$d(x_2,x_3) \leq \mathcal{H}(Tx_1,Tx_2).$$
Thus, we obtain a sequence $\{x_n\}$ in $X$ such that $x_{n+1}\in Tx_n$, $x_n \notin Tx_n$ and
\begin{equation} \label{Alpha-Eq-0}
d(x_{n+1},x_{n+2}) \leq \mathcal{H}(Tx_n,Tx_{n+1}),
\end{equation}
for all $n\in \mathbb{N}$. Since $x_2 \in Tx_1$, $x_3\in Tx_2$ and $T$ is $\alpha$-admissible, we have
\begin{equation*}
\alpha(x_1, x_2) \geq 1 \Longrightarrow \alpha(x_2, x_3) \geq 1.
\end{equation*}
Recursively, we obtain,
\begin{equation}\label{Alpha-Eq-1}
\alpha(x_{n},x_{n+1})\geq 1  \text{ for all } n \in \mathbb{N}\cup\{0\}.
\end{equation}
From \eqref{AlphaAdmGeneralizedZG-contraction},
\begin{align*}
C_G &\leq \zeta \Big(\alpha(x_{n},x_{n+1}) \mathcal{H}(Tx_{n}, Tx_{n+1}), M(x_{n}, x_{n+1})\Big) \\
&< G\Big(M(x_{n}, x_{n+1}), \alpha(x_{n},x_{n+1})\mathcal{H}(Tx_{n}, Tx_{n+1})\Big).
\end{align*}
Further, using (i) of Definition \ref{CGproperty}, we have
\begin{equation}\label{Alpha-Eq-2}
\mathcal{H}(Tx_{n}, Tx_{n+1})\leq\alpha(x_{n},x_{n+1})\mathcal{H}(Tx_{n}, Tx_{n+1}) < M(x_{n}, x_{n+1}),
\end{equation}
where
\begin{align*}
M(x_{n}, x_{n+1})&= \max\Big \lbrace d(x_{n}, x_{n+1}),d(x_{n}, Tx_{n}),d(x_{n+1}, Tx_{n+1}),\\
&~~~~~~~~~~~~~~\frac{d(x_{n}, Tx_{n+1})+d(x_{n+1}, Tx_{n})}{2} \Big \rbrace \\
&=\max \Big \lbrace d(x_n,x_{n+1}), d(x_{n+1}, Tx_{n+1})  \Big \rbrace.
\end{align*}
If $M(x_{n}, x_{n+1})=d(x_{n+1}, Tx_{n+1})$, then \eqref{Alpha-Eq-2} gives
$$\mathcal{H}(Tx_{n}, Tx_{n+1})\leq  d(x_{n+1}, Tx_{n+1}),$$
a contradiction.  Hence $M(x_{n}, x_{n+1}) = d(x_n,x_{n+1})$, and consequently from \eqref{Alpha-Eq-2}, we have
\begin{equation} \label{Alpha-Eq-2A}
d(x_{n+1},x_{n+2})\leq \mathcal{H}(Tx_{n}, Tx_{n+1}) < M(x_{n}, x_{n+1}) =d(x_n,x_{n+1}).
\end{equation}
Hence for all $n \in \mathbb{N}\cup\{0\}$,  we have  $d (x_n, x_{n+1}) > d (x_{n+1}, x_{n+2})$. So $\{d(x_n, x_{n+1})\}$ is a decreasing sequence of nonnegative real numbers, and hence there exists $L\geq 0$ such that $\lim\limits_{n\to \infty}d (x_n, x_{n+1})=\lim\limits_{n\to \infty}M(x_{n}, x_{n+1})=L$.

Assume that $L>0$. Since $\alpha(x_{n},x_{n+1})\mathcal{H}(Tx_{n}, Tx_{n+1}) < M(x_{n}, x_{n+1})$, so we get
\begin{equation}\label{Alpha-Eq-3}
\lim\limits_{n\to\infty}\alpha(x_{n},x_{n+1})\mathcal{H}(Tx_{n}, Tx_{n+1})=L.
\end{equation}
Then using \eqref{AlphaAdmGeneralizedZG-contraction} and (b) of Definition \ref{CGSimulationfunction}, we get
\begin{align*}
C_G &\leq \lim\limits_{n\to\infty}\sup \zeta\Big(\alpha(x_{n},x_{n+1})\mathcal{H}(Tx_{n}, Tx_{n+1}), M(x_{n}, x_{n+1})\Big)\\
&=\lim\limits_{n\to\infty}\sup \zeta\Big(\alpha(x_{n},x_{n+1})\mathcal{H}(Tx_{n}, Tx_{n+1}), d(x_{n}, x_{n+1})\Big) < C_G,
\end{align*}
which is a contradiction and hence $L=0$.

Now we show that $\{x_n\}$ is a Cauchy sequence. If not, then by Lemma \ref{LemmaA}  we have
\begin{equation}\label{Alpha-Eq-4-i}
\lim\limits_{k\to\infty} d(x_{m(k)},x_{n(k)})=\lim\limits_{k\to\infty} d(x_{m(k)+1},x_{n(k)+1})=\varepsilon
\end{equation}
and consequently,
\begin{equation}\label{Alpha-Eq-4-ii}
\lim\limits_{k\to\infty}M(x_{m(k)}, x_{n(k)})=\varepsilon.
\end{equation}
Let $x=x_{m(k)}, y=x_{n(k)}$.
Since $T$ is triangular $\alpha$-orbital admissible, so by Lemma \ref{Lemma3a}, we have
$\alpha(x_{m(k)},x_{n(k)}) \geq 1$.
Then by \eqref{AlphaAdmGeneralizedZG-contraction},
\begin{align*}
C_G & \leq \zeta \Big( \alpha(x_{m(k)},x_{n(k)}) \mathcal{H}(Tx_{m(k)},Tx_{n(k)}), M(x_{m(k)},x_{n(k)} )  \Big)\\
& < G \Big( M(x_{m(k)},x_{n(k)}), \alpha(x_{m(k)},x_{n(k)}) \mathcal{H}(Tx_{m(k)},Tx_{n(k)})  \Big).
\end{align*}
Here $M(x_{m(k)},x_{n(k)})=d(x_{m(k)},x_{n(k)})$, so further by (i) of Definition \ref{CGproperty}, we get
\begin{equation}\label{Alpha-Eq-4}
\begin{split}
d(x_{m(k)+1},x_{n(k)+1}) \leq \alpha(x_{m(k)},x_{n(k)}) \mathcal{H}(Tx_{m(k)},Tx_{n(k)})\\ < M(x_{m(k)},x_{n(k)}) = d(x_{m(k)},x_{n(k)}).
\end{split}
\end{equation}
Using \eqref{Alpha-Eq-4-i} and \eqref{Alpha-Eq-4-ii} in \eqref{Alpha-Eq-4}, we get
\begin{equation*}
\lim\limits_{k\to\infty}\alpha(x_{m(k)},x_{n(k)}) \mathcal{H}(Tx_{m(k)},Tx_{n(k)})=\varepsilon.
\end{equation*}
Therefore using \eqref{AlphaAdmGeneralizedZG-contraction} and (b) of Definition \ref{CGSimulationfunction}, we get
$$C_G \leq \lim\limits_{n\to\infty}\sup \zeta\Big(\alpha(x_{m(k)},x_{n(k)}) \mathcal{H}(Tx_{m(k)},Tx_{n(k)}), M(x_{m(k)}, x_{n(k)})\Big) < C_G,$$ which is a contradiction. Hence $\{x_n\}$ is a Cauchy sequence.
From \eqref{Alpha-Eq-1} and the $\alpha$-completeness of $(X, d)$, there exists $u \in X$ such that $x_n \xrightarrow[]{d} u$ as $n \to \infty$.

By $\alpha$-continuity of the multivalued mapping $T$, we get
\begin{equation}\label{Alpha-Eq-4a}
\lim\limits_{n\to\infty}\mathcal{H}(Tx_n,Tx)=0.
\end{equation}
Thus we obtain
\begin{equation*}
d(u,Tu)= \lim\limits_{n\to \infty}d(x_{n+1}, Tu) \leq \lim\limits_{n\to \infty}\mathcal{H}(Tx_{n}, Tu) =0.
\end{equation*}
Therefore, $u \in Tu$ and hence $T$ has a fixed point.
\end{proof}

\begin{theorem} \label{TheoremAlpha2}	
Let $(X, d)$ be a metric space and $T : X \to \mathcal{CB}(X)$ be an $\alpha$-admissible $\mathcal{Z}_G$-contractive multivalued mapping. Suppose following conditions hold:
\begin{enumerate}
\item[(i)] $(X, d)$ is an $\alpha$-complete metric space,
\item[(ii)] there exist $x_0 \in X$ and $x_1\in Tx_0$ such that $\alpha(x_0,x_1) \geq 1$,
\item[(iii)] $T$ is a triangular $\alpha$-admissible,
\item[(iv)] $T$ is an $\alpha$-continuous multivalued mapping.
\end{enumerate}
Then $T$ has a fixed point.
\end{theorem}
\begin{proof}
The proof follows in the same manner as in Theorem \ref{TheoremAlpha1}.
\end{proof}

\begin{corollary} 
Let $(X, d)$ be a metric space, $T : X \to \mathcal{CB}(X)$ be a  $\alpha$-admissible generalized $\mathcal{Z}_G$-contractive (or, $\alpha$-admissible $\mathcal{Z}_G$-contractive) multivalued mapping. Suppose following conditions hold:
\begin{enumerate}
\item[(i)] $(X, d)$ is an $\alpha$-complete metric space,
\item[(ii)] there exist $x_0 \in X$ and $x_1\in Tx_0$ such that $\alpha(x_0,x_1) \geq 1$,
\item[(iii)] $T$ is a triangular $\alpha_*$-admissible,
\item[(iv)] $T$ is an $\alpha$-continuous multivalued mapping.
\end{enumerate}
Then $T$ has a fixed point.
\end{corollary}

\begin{corollary} 
Let $(X, d)$ be a metric space, $T : X \to \mathcal{CB}(X)$ be a  $\alpha$-admissible generalized $\mathcal{Z}_G$-contractive (or, $\alpha$-admissible $\mathcal{Z}_G$-contractive)  multivalued mapping. Suppose following conditions hold:
\begin{enumerate}
\item[(i)] $(X, d)$ is complete metric space,
\item[(ii)] there exist $x_0 \in X$ and $x_1\in Tx_0$ such that $\alpha(x_0,x_1) \geq 1$,
\item[(iii)] $T$ is a triangular $\alpha_*$-admissible,
\item[(iv)] $T$ is a continuous multivalued mapping.
\end{enumerate}
Then $T$ has a fixed point.
\end{corollary}

\begin{example}\rm \label{Example1}
Let $X=(-10,10)$ with the metric $d(x,y)= |x-y|$ and  $T : X \to \mathcal{CB}(X)$ be defined as:
$$T(x) = \left\{\begin{array}{lll}
\{-2\} ~& \mbox{if} ~x \in (-10,0), \\&\\

\left[0,\frac{5x}{6}\right] ~& \mbox{if} ~x \in [0,2],\\&\\

\left[0, 2x-\frac{5}{3}\right] ~& \mbox{if} ~x \in (2,5],\\&\\

\{9\} ~& \mbox{if} ~x \in (5,10).
\end{array}\right.$$ 
Define $ \alpha: X\times X \to [0,\infty)$ by
$$\alpha(x,y) = \left\{\begin{array}{lll}
1 ~& \mbox{if} ~x,y \in [0,2], \\
&\\
0 ~&   \mbox{otherwise}.
\end{array}\right.$$
Then the space $(X,d)$ is $\alpha$-complete and $T$ is not continuous but it is $\alpha$-continuous. Also $T$ is an triangular $\alpha$-admissible multivalued mapping, since if $\alpha(x,y)\geq 1$, then we have $x,y \in [0,2]$, and so $Tx, Ty \subseteq \left[0,\frac{5}{3}\right]$, which implies $\alpha(p,q)\geq 1$ for all $p\in Tx$ and $q\in Ty$. Thus, $T$ is $\alpha$-admissible. Further, if $\alpha(x,y) \geq 1$ then $x,y \in [0,2]$. So $x\in [0,2]$ and $Ty \subseteq \left[0, \frac{5}{3}\right]$. Let $z\in Ty$. Then we have $\alpha(y,z) \geq 1$. Finally, $x\in [0,2]$ and $z\in \left[0, \frac{5}{3}\right]$ gives $\alpha(x,z)\geq 1$. Hence $T$ is triangular $\alpha$-admissible.

If we choose $x_0=2$ then condition (ii) of Theorem \ref{TheoremAlpha1} holds.
%
%
Consider $\zeta(t,s)=\frac{5}{6}s-t$ and $G(s,t)=s-t$, then $T$ is an $\alpha$-admissible generalized $\mathcal{Z}_G$-contractive multivalued mapping with $C_G=0$.
Thus, all the conditions of Theorem \ref{TheoremAlpha1} are satisfied. Consequently $T$ has fixed points in $X$.
\end{example}

The next results show that the $\alpha$-continuity or continuity of the mapping $T$ can be relaxed by assuming the condition $(iv')$ as follows.

\begin{theorem} \label{TheoremAlpha3}
Let $(X, d)$ be a metric space, $T : X \to \mathcal{CB}
(X)$ be a  $\alpha$-admissible generalized $\mathcal{Z}_G$-contractive multivalued mapping. Suppose following conditions hold:
\begin{enumerate}
\item[(i)] $(X, d)$ is an $\alpha$-complete metric space,
\item[(ii)] there exist $x_0 \in X$ and $x_1\in Tx_0$ such that $\alpha(x_0,x_1) \geq 1$,
\item[(iii)] $T$ is a triangular $\alpha$-admissible,
\item[(iv')] if $\{x_n\}$ is a sequence in $X$ such that $\alpha(x_n,x_{n+1}) \geq 1$ for all $n\in \mathbb{N}$ and $x_n \xrightarrow[]{d} x \in X$ as $n \to \infty$, then we have $\alpha(x_{n},x) \geq 1$ for all $n \in \mathbb{N}$.
\end{enumerate}
Then $T$ has a fixed point.
\end{theorem}

\begin{proof}
Following the proof of Theorem \ref{TheoremAlpha1}, we know that $\{x_n\}$ is a Cauchy sequence in $X$ such that $x_n \xrightarrow[]{d} x\in X$ as $n \to \infty$ and $\alpha(x_n, x_{n+1})\geq 1$ for all $n \in \mathbb{N}$. 

From condition $(iv')$, we get $\alpha(x_n,u)\geq 1$ for all $n \in \mathbb{N}$. By using \eqref{AlphaAdmGeneralizedZG-contraction}, we have
\begin{equation}\label{Alpha-Eq-4b}
 \zeta\big(\alpha(x_n, u) \mathcal{H}(Tx_n, Tu), M(x_n, u)\big)\geq C_G
\end{equation}
where
\begin{equation*}
M(x_n, u)=\max\left\lbrace d(x_n,u), d(x,Tx_n), d(u,Tu), \frac{d(x_n,Tu)+d(u,Tx_n)}{2} \right\rbrace.
\end{equation*}
for all $n \in \mathbb{N}$. Suppose that $d(u,Tu)>0$. Let $\varepsilon = \frac{d(u,Tu)}{2}$. Since $x_n \xrightarrow[]{d} x\in X$ as $n \to \infty$, we can find $n_1\in \mathbb{N}$ such that 
\begin{equation}  \label{Alpha-Eq-4c1}
d(u,x_n) < \frac{d(u,Tu)}{2}
\end{equation}
for all $n\geq n_1$. Furthermore, we obtain that
\begin{equation}  \label{Alpha-Eq-4c2}
d(u,Tx_n) \leq d(u, x_{n+1})< \frac{d(u,Tu)}{2}
\end{equation}
for all $n\geq n_1$. Also, as $\{x_n\}$ is a Cauchy sequence, there exists $n_2 \in \mathbb{N}$ such that 
\begin{equation}  \label{Alpha-Eq-4c3}
d(x_n,Tx_n) \leq d(x_n,x_{n+1})< \frac{d(u,Tu)}{2}
\end{equation}
for all $n\geq n_2$. It follows from $d(x_n,Tu) \to d(u,Tu)$ as $n \to \infty$ that we can find $n_3 \in \mathbb{N}$ such that 
\begin{equation}   \label{Alpha-Eq-4c4}
d(x_n, Tu) < \frac{3d(u,Tu)}{2}
\end{equation}
for all $n\geq n_3$. Thus, using \eqref{Alpha-Eq-4c1}-\eqref{Alpha-Eq-4c4}, we get
\begin{equation} \label{Alpha-Eq-4c5}
M(x_n, u)=d(u,Tu)
\end{equation}
for all $n\geq n_0=\max\{n_1,n_2,n_3\}$. Hence from \eqref{Alpha-Eq-4b}, we have
\begin{equation*}\label{Alpha-Eq-4d}
C_G  \leq \zeta\Big(\alpha(x_n, u) \mathcal{H}(Tx_n, Tu), d(u,Tu)\Big) < G\Big(d(u,Tu), \alpha(x_n, u) \mathcal{H}(Tx_n, Tu)\Big). 
\end{equation*}
Using (i) of Definition \ref{CGproperty}, we get $$\alpha(x_n, u) \mathcal{H}(Tx_n, Tu)< d(u,Tu).$$ Further, since
\begin{equation} \label{Alpha-Eq-4d}
d(x_{n+1},Tu) \leq \mathcal{H}(Tx_n, Tu) \leq \alpha(x_n, u) \mathcal{H}(Tx_n, Tu)< d(u,Tu),
\end{equation}
letting $n \to \infty$, we get $d(u,Tu)<d(u,Tu)$, which is a contradiction. Therefore, $d(u,Tu)=0$, that is,  $u\in Tu$. This completes the proof.
\end{proof}

\begin{corollary} 
Let $(X, d)$ be a metric space, $T : X \to \mathcal{CB}(X)$ be a  $\alpha$-admissible generalized $\mathcal{Z}_G$-contractive (or, $\alpha$-admissible $\mathcal{Z}_G$-contractive) multivalued mapping. Suppose following conditions hold:
\begin{enumerate}
\item[(i)] $(X, d)$ is an $\alpha$-complete metric space,
\item[(ii)] there exist $x_0 \in X$ and $x_1\in Tx_0$ such that $\alpha(x_0,x_1) \geq 1$,
\item[(iii)] $T$ is a triangular $\alpha_*$-admissible,
\item[(iv')] if $\{x_n\}$ is a sequence in $X$ such that $\alpha(x_n,x_{n+1}) \geq 1$ for all $n\in \mathbb{N}$ and $x_n \xrightarrow[]{d} x \in X$ as $n \to \infty$, then we have $\alpha(x_{n},x) \geq 1$ for all $n \in \mathbb{N}$.
\end{enumerate}
Then $T$ has a fixed point.
\end{corollary}

\begin{corollary} 
Let $(X, d)$ be a metric space, $T : X \to \mathcal{CB}(X)$ be a  $\alpha$-admissible generalized $\mathcal{Z}_G$-contractive (or, $\alpha$-admissible $\mathcal{Z}_G$-contractive)  multivalued mapping. Suppose following conditions hold:
\begin{enumerate}
\item[(i)] $(X, d)$ is complete metric space,
\item[(ii)] there exist $x_0 \in X$ and $x_1\in Tx_0$ such that $\alpha(x_0,x_1) \geq 1$,
\item[(iii)] $T$ is a triangular $\alpha_*$-admissible,
\item[(iv')] if $\{x_n\}$ is a sequence in $X$ such that $\alpha(x_n,x_{n+1}) \geq 1$ for all $n\in \mathbb{N}$ and $x_n \xrightarrow[]{d} x \in X$ as $n \to \infty$, then we have $\alpha(x_{n},x) \geq 1$ for all $n \in \mathbb{N}$.
\end{enumerate}
Then $T$ has a fixed point.
\end{corollary}

\begin{example}\rm \label{Example2}
Let $X=(0,1]$ with the metric $d(x,y)= |x-y|$ and  $T : X \to \mathcal{CB}(X)$ be defined as:
	$$Tx = \left\{\begin{array}{lll}
	\{\frac{1}{10}\} ~& \mbox{if} ~x \in (0,\frac{1}{2}), \\&\\
	
	\{\frac{3}{5}, \frac{3}{4}\}  ~& \mbox{if} ~x \in [\frac{1}{2},\frac{3}{4}],\\&\\
	
	\{\frac{4}{5}\} ~& \mbox{if} ~x \in (\frac{3}{4},1].
	\end{array}\right.$$ 
	Define $ \alpha: X\times X \to [0,\infty)$ by
	$$\alpha(x,y) = \left\{\begin{array}{lll}
	1 ~& \mbox{if} ~x,y \in [\frac{1}{2},1], \\
	&\\
	0 ~&   \mbox{otherwise}.
	\end{array}\right.$$
Then $(X,d)$ is $\alpha$-complete metric space. The mapping $T$ is not $\alpha$-continuous (to see this consider $x_n= \frac{3}{4}+\frac{1}{n}$, $x= \frac{3}{4}$). Also, $T$ is triangular $\alpha$-admissible. Further, there exists $x_0= \frac{1}{2}$, and $x_1= \frac{3}{4} \in Tx_0= \{\frac{3}{5}, \frac{3}{4}\}$ such that $\alpha(x_0,x_1)\geq 1$.
	
Now, let $\{x_n\}$ be any sequence in $X$ such that $\alpha(x_n,x_{n+1}) \geq 1$ for all $n \in \mathbb{N}$ and $x_n \to x$ as $ n \to \infty$. This implies $ x_n \in \left[\frac{1}{2},1\right]$ for every $n \in \mathbb{N}$ and hence $x \in \left[\frac{1}{2},1\right]$. Thus $\alpha(x_n,x) \geq 1$ and so condition $(iv')$ is satisfied. One can easily verify that $T$ is an $\alpha$-admissible $\mathcal{Z}_G$-contractive multivalued mapping by taking $\zeta(t,s)=\frac{5}{6}s-t$, $G(s,t)=s-t$ and $C_G=0$. Thus, $T$ satisfies all the conditions of Theorem \ref{TheoremAlpha3} having some fixed points.
\end{example}

\section{Consequences}

In 2008, Jachymski \cite{Jachymski2008} using the language of graph theory (which subsumes the partial ordering) introduced the concept of G-contraction on a metric space endowed with a graph and proved a fixed point theorem which extends the results of Ran and Reurings \cite{RanReurings2004}. Afterwards, some results of Jachymski \cite{Jachymski2008} have been extended to multivalued mappings in \cite{DinevariFrigon2013,BegButt2013,ChifuPetrusel2012,NicolaeOReganPetrusel2011}.

In this section, we give fixed point results on a metric space endowed with a graph.
Before presenting our results, we give the following notions and definitions.

Let $(X,d)$ be a metric space and $\Delta = \{(x,x) : x \in X\}$. Consider a graph $G$ with the set $V(G)$ of its vertices equal to $X$ and the set $E(G)$ of its edges as a superset of $\Delta$. Assume that $G$ has no parallel edges, that is $(x,y), (y,x) \in E(G)$ implies $x=y$. Also, $G$ is directed if the edges have a direction associated with them. Now we can identify the graph $G$ with the pair $\left(V(G), E(G)\right)$. Moreover, we may treat G as a weighted graph by assigning to each edge
the distance between its vertices.

\begin{definition} \rm 
Let $X$ be a nonempty set endowed with a graph $G$ and $T: X\to \mathcal{N}(X)$ be a multivalued mapping. Then $T$ is said to be triangular edge preserving if for each $x\in X$ and $y\in Tx$ with $(x,y), (y,z) \in E(G)$, we have $(x,z)\in E(G)$ for all $z\in Ty$.
\end{definition}

\begin{definition} \rm 
Let $(X, d)$ be a metric space endowed with a graph $G$. The metric space $X$
is said to be $E(G)$-complete if and only if every Cauchy sequence $\{x_n\}$ in $X$ with $(x_n, x_{n+1}) \in E(G)$ for all $n \in \mathbb{N}$, converges in $X$.
\end{definition}

\begin{definition} \rm 
Let $(X, d)$ be a metric space endowed with a graph $G$. We say that $T : X \to \mathcal{CB}(X)$ is an $E(G)$-continuous mapping to $(\mathcal{CB}(X), \mathcal{H})$ if for given $x\in X$ and sequence $\{x_n\}$ with 
$$\lim\limits_{n\to \infty} d(x_n,x)=0~~~\text{and}~~(x_n,x_{n+1})\in E(G)~~\text{for all}~n \in \mathbb{N}~\Longrightarrow ~~\lim\limits_{n \to \infty}\mathcal{H}(Tx_n, Tx) =0.$$
\end{definition}

\begin{definition} \rm 
Let $(X, d)$ be a metric space endowed with a graph $G$. A mapping $T: X\to \mathcal{CB}(X)$ is said to be a $E(G)$-$\mathcal{Z}_G$-contractive  mapping if there exist $\zeta\in \mathcal{Z}_G$ and $\alpha : X \times X \to  [0,\infty)$ such that
\begin{equation}\label{alphaAdmZGContraction}
x,y\in X, ~~(x,y)\in E(G) ~~ \Longrightarrow~~ \zeta\Big(\alpha(x, y) \mathcal{H}(Tx, Ty), d(x, y)\Big)\geq C_G.
\end{equation}
Similarly, by taking $M(x,y)$ instead of $d(x,y)$, we can define the generalized $E(G)$-$\mathcal{Z}_G$-contractive  mapping.
\end{definition}

\begin{theorem} 	
Let $(X, d)$ be a metric space endowed with a graph $G$, and $T : X \to \mathcal{CB}(X)$ be an $E(G)$-$\mathcal{Z}_G$-contractive mapping. Suppose following conditions hold:
\begin{enumerate}
\item[(i)] $(X, d)$ is an $E(G)$-complete metric space,
\item[(ii)] there exist $x_0 \in X$ and $x_1\in Tx_0$ such that $(x_0,x_1) \in E(G)$,
\item[(iii)] $T$ is triangular edge preserving,
\item[(iv)] $T$ is an $E(G)$-continuous multivalued mapping.
\end{enumerate}
Then $T$ has a fixed point.
\end{theorem}

\begin{proof}
This result can be obtained from Theorem \ref{TheoremAlpha2} by defining a mapping $\alpha : X \times X \to [0,\infty)$ such that
$$\alpha(x,y) = \left\{\begin{array}{lll}
1 ~& \mbox{if} ~(x,y) \in E(G), \\
&\\
0 ~&   \mbox{otherwise}.
\end{array}\right.$$
This completes the proof.
\end{proof}
By using Theorem \ref{TheoremAlpha3}, we get the following result.

\begin{theorem} 
Let $(X, d)$ be a metric space endowed with a graph $G$, and $T : X \to \mathcal{CB}(X)$ be an $E(G)$-$\mathcal{Z}_G$-contractive mapping. Suppose following conditions hold:
\begin{enumerate}
\item[(i)] $(X, d)$ is an $E(G)$-complete metric space,
\item[(ii)] there exist $x_0 \in X$ and $x_1\in Tx_0$ such that $(x_0,x_1) \in E(G)$,
\item[(iii)] $T$ is triangular edge preserving,
\item[(iv')] if $\{x_n\}$ is a sequence in $X$ such that $(x_n,x_{n+1}) \in E(G)$ for all $n\in \mathbb{N}$ and $x_n \rightarrow x \in X$ as $n \to \infty$, then we have $(x_{n},x) \in E(G)$ for all $n \in \mathbb{N}$.
\end{enumerate}
Then $T$ has a fixed point.
\end{theorem}


\begin{thebibliography}{99}

\bibitem{AbbasLatif2016} Abbas M., Latif A., Suleiman Y., \textit{Fixed points for cyclic $R$-contractions and solution of nonlinear Volterra integro-differential equations}, Fixed Point Theory Appl. 2016, Article ID 61, (2016).

\bibitem{AlikhaniGopalMiandaraghRezapourShahzad2013} Alikhani H., Gopal D., Miandaragh M.A., Rezapour Sh., Shahzad N., \textit{Some endpoint results for $\beta$-generalized weak contractive multifunctions}, The Scientific World Journal 2013 , Article ID 948472, (2013).

\bibitem{Ansari2014} Ansari A.H., \textit{Note on  $\varphi$-$\psi$-contractive type mappings and related fixed point}, The 2nd Regional Conference on Math. Appl. PNU, Sept. 2014, 377--380, (2014).

\bibitem{ArgoubiSametVetro2015} Argoubi H., Samet B., Vetro C., \textit{Nonlinear contractions involving simulation functions in a metric space with a partial order}, J. Nonlinear Sci. Appl. 8, 1082--1094, (2015).

\bibitem{AslRezapourShahzad2012} Asl J.H., Rezapour S., Shahzad N, \textit{On fixed points of $\alpha$-$\psi$-contractive multifunctions}, Fixed Point Theory Appl. 2012, Article ID 212, (2012).

\bibitem{AydiAbbasVetro2012} Aydi H., Abbas M., Vetro C., \textit{Partial Hausdorff metric and Nadler's fixed point theorem on partial metric spaces}, Topology and its Applications 159(14), 3234-3242, (2012).


\bibitem{BegButt2013} Beg I., Butt A.R., \textit{Fixed point of set-valued graph contractive mappings}, J. Inequal. Appl. 2013, Article ID 252, (2013).

\bibitem{BerindeBerinde2007} Berinde M., Berinde V., \textit{On a general class of multi-valued weakly Picard mappings}, Journal of Mathematical Analysis and Applications, Volume 326(2), 772-782, (2007).

\bibitem{ChenTang2016} Chen J., Tang X., \textit{Generalizations of Darbos fixed point theorem via simulation functions with application to functional integral equations}, J. Comput. Appl. Math. 296, 564--575, (2016). 

\bibitem{ChifuPetrusel2012} Chifu C., Petrusel G., \textit{Generalized contractions in metric spaces endowed with a graph}, Fixed Point Theory Appl. 2012, Article ID 161, (2012).

\bibitem{Ciric2009} Ciric L., \textit{Multi-valued nonlinear contraction mappings, Nonlinear Analysis: Theory}, Methods and Applications  71(7-8), 2716-2723, (2009).

\bibitem{CovitzNadler1970} Covitz H., Nadler S.B., \textit{Multi-valued contraction mappings in generalized metric spaces}, Israel Journal of Mathematics 8(1), 5-11, (1970).

\bibitem{DinevariFrigon2013} Dinevari T., Frigon M., \textit{Fixed point results for multivalued contractions on a metric space with a graph}, J. Math. Anal. Appl. 405, 507-517, (2013).


\bibitem{HierroKarapinar} de-Hierro A.-F. R.-L., Karapinar E., de-Hierro C. R.-L., M.-Moreno J., \textit{Coincidence point theorems on metric spaces via simulation functions}, J. Comput. Appl. Math. 275, 345--355, (2015).

\bibitem{HieroSamet} de-Hierro A.-F. R.-L., Samet B., \textit{$\varphi$-admissibility results via extended simulation functions}, J. Fixed Point Theory Appl., DOI 10.1007/s11784-016-0385.

\bibitem{HussainKutbiSalimi2014} Hussain N., Kutbi M.A., Salimi P., \textit{Fixed point Theory in $\alpha$-complete metric spaces with applications}, Abstract and Applied Analysis 2014, Article ID 280817, (2014).

\bibitem{Jachymski2008} Jachymski J., \textit{The contraction principle for mappings on a metric space with a graph}, Proc. Amer. Math. Soc. 136(4), 1359--1373, (2008).

\bibitem{Karapinar2016} Karapinar E., \textit{Fixed points results via simulation functions}, Filomat 30(8), 2343--2350, (2016).

\bibitem{KarapinarKhojasteh} Karapinar E., Khojasteh F., \textit{An approach to best proximity points results via simulation functions}, J. Fixed Point Theory Appl., DOI 10. 1007/s11784-016-0380-2.


\bibitem{Karapinar2012} Karapinar E., Samet B., \textit{Generalized $(\alpha$-$\psi)$-contractive type mappings and related fixed point theorems with applications}, Abstr. Appl. Anal, Vol. 2012, Article ID 793486, (2012).


\bibitem{KarapinarKumamSalimi2013} Karapinar E., Kumam P., Salimi P., \textit{On $\alpha$-$\psi$-Meir-Keeler contractive mappings}, Fixed Point Theory and Applications 2013, Article ID 94, (2013).

\bibitem{KhojastehShuklaRadenovic} Khojasteh F., Shukla S., Radenovic S., \textit{A new approach to the study of fixed point theorems via simulation functions}, Filomat, 29(6), 1189--1194, (2015).

\bibitem{KlimWardowski2007} Klim D., Wardowski D., \textit{Fixed point theorems for set-valued contractions in complete metric spaces}, Journal of Mathematical Analysis and Applications 334(1), 132-139, (2007).

\bibitem{KomalKumamGopal2016} Komal S., Kumam P., Gopal D., \textit{Best proximity point for $\mathcal{Z}$-contraction and Suzuki type $\mathcal{Z}$-contraction mappings with an application to fractional calculus}, Appl. Gen. Topol. 17(2), 185-198, (2016).

\bibitem{KumamGopalBudhia2016} Kumam P., Gopal D., Budhia L., \textit{A new fixed point theorem under Suzuki type $\mathcal{Z}$-contraction mappings}, Journal of Mathematical Analysis 8(1), 113-119, (2016).


\bibitem{KutbiSintunavarat2015} Kutbi M.A., Sintunavarat W., \textit{On new fixed point results for $(\alpha,\psi, \zeta)$-contractive multivalued mappings on $\alpha$-complete metric spaces and their consequences}, Fixed Point Theory and Applications 2015, Article ID 2, (2015).

\bibitem{LiuAnsariChandokRadenovic2018} Liu X-L, Ansari A.H., Chandok S., Radenovi\'{c} S., \textit{On some results in metric spaces using auxiliary simulation functions via new functions}, J. Computational Analysis And Applications 24(6), 1103-1114, (2018).

\bibitem{Markin1968} Markin J.T., \textit{A fixed point theorem for set valued mappings}, Bull. Amer. Math. Soci. 74, 639--640, (1968).


\bibitem{MinakAcarAltun2013} Minak G., Acar O., Altun I., \textit{Multivalued pseudo-Picard operators and fixed point results}, J. Funct. Spaces Appl. 2013, Article ID 827458, (2013).

\bibitem{MizoguchiTakahashi1989} Mizoguchi N., Takahashi W., \textit{Fixed point theorems for multivalued mappings on complete metric spaces}, Journal of Mathematical Analysis and Applications 141(1), 177-188, (1989).

\bibitem{MohammadiRezapourNaseer2013} Mohammadi B., Rezapour S., Shahzad N., \textit{Some results on fixed points of $\alpha$-$\psi$-Ciric generalized multifunctions}, Fixed Point Theory Appl. 2013, Article ID 24, (2013).

\bibitem{MongkolkehaChoKumam2017} Mongkolkeha C., Cho Y.J., Kumam P., \textit{Fixed point theorems for simulation functions in $b$-metric spaces via the $wt$-distance}, Appl. Gen. Topol. 18(1), 91-105, (2017).

\bibitem{Nadler1969} Nadler S.B., \textit{\lq\lq Multivalued contraction mappings\rq\rq}, Pacific Journal of Mathematics 30, 475-488, (1969).

\bibitem{NastasiVetro2015} Nastasi A., Vetro P., \textit{Fixed point results on metric and partial metric spaces via simulation functions}, J. Nonlinear Sci. Appl. 8, 1059--1069, (2015).

\bibitem{NicolaeOReganPetrusel2011} Nicolae A., \'O Regan, D., Petrusel A., \textit{Fixed point theorems for singlevalued and multivalued generalized contractions in metric spaces endowed with a graph}, Georgian Math. J. 18, 307-327, (2011).

\bibitem{OlgunBicer2016} Olgun M., Bicer O., Alyildiz T., \textit{A new aspect to Picard operators with simulation functions}, Turk. J. Math. 40, 832--837, (2016).

\bibitem{Popescu2014} Popescu O., \textit{Some new fixed point theorems for $\alpha$-Geraghty contraction type maps in metric spaces}, Fixed Point Theory and Applications 2014, Article ID 190, (2014).

\bibitem{RadenovicChandok} Radenovic S., Chandok S., \textit{Simulation type functions and coincidence point results}, Filomat (to appear).

\bibitem{RanReurings2004} Ran A.C.M., Reurings M.C.B., \textit{A fixed point theorem in partially ordered sets and some applications to matrix equations}, Proc. Amer. Math. Soc. 132(5), 1435--1443, (2004).

\bibitem{Samet2015} Samet B., \textit{Best proximity point results in partially ordered metric spaces via simulation functions}, Fixed Point Theory Appl. 2015, Article ID 232, (2015).

\bibitem{SametVetroVetro2012} Samet B., Vetro C., Vetro P., \textit{Fixed point theorem for $\alpha$-$\psi$ Contractive type mappings}, Nonlinear Anal. 75, 2154--2165, (2012).

\bibitem{ShuklaGopalMoreno2017} Shukla S., Gopal D., Mart\'{i}nez-Moreno J., \textit{Fixed Points of set-valued $F$-contractions and its Application to non-linear integral equations}, Filomat 31(11), 3377-3390, (2017).

\bibitem{Suzuki2008JMAA} Suzuki T., \textit{Mizoguchi-Takahashi's fixed point theorem is a real generalization of Nadler's}, Journal of Mathematical Analysis and Applications 340(1), 752-755, (2008).

\bibitem{TchierVetro2016} Tchier F., Vetro C., Vetro F., \textit{Best approximation and variational inequality problems involving a simulation function}, Fixed Point Theory Appl. 2016, Article ID 26, (2016).

\end{thebibliography}
\end{document}